\documentclass[a4paper,11pt]{article}

\usepackage{a4wide}
\usepackage[english]{babel}

\usepackage{amsfonts}
\usepackage{amsmath}
\usepackage{graphicx}
\usepackage[colorlinks,bookmarks=true]{hyperref}
\usepackage{amsthm}
\usepackage{enumitem}

\newtheorem{lemma}{Lemma}[section]

\newtheorem{theorem}{Theorem}[section]
\newtheorem{proposition}{Proposition}[section]

\newtheorem{assumption}{Assumption}

\newtheorem{corollary}{Corollary}[section]

\newcommand{\E}{\mathbb{E}}

\title{The different asymptotic regimes of nearly unstable autoregressive processes}

\author{Thibault Jaisson\\ CMAP, \'Ecole Polytechnique Paris \\ thibault.jaisson@polytechnique.edu
\\$~~$\\Mathieu Rosenbaum\\ LPMA, Universit\'e Pierre et Marie Curie (Paris 6)\\mathieu.rosenbaum@upmc.fr
}

\begin{document}

\maketitle

\begin{abstract}
\noindent We extend the results of \cite{cumberland1982weak,phillips1987towards,van1999asymptotic} about the convergence of nearly unstable AR($p$) processes to the infinite order case. To do so, we proceed as in \cite{jaisson2013limit,jaisson2014fractional} by using limit theorems for some well chosen geometric sums. We prove that when the coefficients sequence has a light tail, nearly unstable AR($\infty$) processes behave as Ornstein-Uhlenbeck models. However, in the heavy tail case, we show that fractional diffusions arise as limiting laws for such processes.
\end{abstract}


\section{Introduction}

\noindent In the field of time series analysis, autoregressive processes (AR processes for short) probably represent the most classical class of models. 
In this work, a discrete time process $y$ is said to be autoregressive if it satisfies $y_0=\varepsilon_0$ and for $k\geq 1$,
$$y_k=\varepsilon_k+\sum_{i=1}^{k}\phi_i y_{k-i},$$
where $\phi$ is a sequence of non-negative coefficients and the $\varepsilon_k$ are iid centered random variables with finite second order moment. The process is called AR($\infty$) if an infinite number of coefficients are non-zero and AR($p$) if they are all equal to zero after rank $p$. A first reason for the popularity of AR processes is the fact that they are quite tractable. Furthermore, it is usually easy to give an interpretation for such dynamics in practice. That is why they are used in various fields such as population dynamics, see \cite{cumberland1982weak,royama1992analytical}, finance, see \cite{hasbrouck1991measuring}, or telecommunications, see \cite{adas1997traffic}\footnote{Of course the references given here are just some examples among many others.}.\\

\noindent For classical AR processes starting at $-\infty$ and with non-negative coefficients, it is well known that a necessary condition for the existence of a stationary solution is
\begin{equation}\label{stabcond}
\sum_{i=1}^{+\infty}\phi_i<1.
\end{equation}
Thus, \eqref{stabcond} is called stability condition. However, in many applications, notably finance, the case where this stability condition is almost saturated: $$0\leq 1-\sum_{i=1}^{+\infty}\phi_i\ll 1$$ seems to be relevant, see \cite{cumberland1982weak,markevivciute2013asymptotic}.\\

\noindent To address this near instability situation, several authors have considered the framework of sequences of AR(1) processes of the form 
$$y^n_k=\rho_n y^n_{k-1}+\varepsilon_k,$$
with the autoregression parameter $\rho_n$ tending to one as $n$ goes to infinity, see for example \cite{cumberland1982weak,markevivciute2012functional,phillips1987towards}. In these works, it is shown that after suitable renormalization, such processes asymptotically behave as Ornstein-Uhlenbeck models. The estimation of the AR coefficient has also been extensively studied in such an asymptotic, see \cite{chan1988parameter,chan1987asymptotic} and \cite{buchmann2007asymptotic,chan2009time,chan2009inference} for more recent related developments. The situation where the number of coefficients in the autoregressive process is larger than one is treated in \cite{buchmann2013unified,jeganathan1991asymptotic,van1999asymptotic}. However, to our knowledge, no result is available in the case of nearly unstable AR($\infty$) processes. The aim of this work is to fill this gap, using a methodology originally developed for Hawkes processes in \cite{jaisson2013limit,jaisson2014fractional}.\\

\noindent A Hawkes process $(N_t)$ is a point process whose intensity at time $t$, denoted by $\lambda_t$, is of the form
$$\lambda_t=\mu+\int_{-\infty}^t\phi(t-s)dN_s,$$
with $\mu$ a constant and $\phi$ a non-negative measurable function, see \cite{hawkes1971point}. The linearity in the previous equation implies that Hawkes processes have a common structure with autoregressive processes. In particular, under the stability condition that the $L^1$ norm of $\phi$, denoted by $|\phi|_{L^1}$, is strictly smaller than one, Hawkes processes admit a stationary intensity. Let us also mention that the counterpart of the Yule-Walker equations linking the autocorrelations of an AR process to its coefficients is the Wiener-Hopf equation, see \cite{bacry2014second,hawkes1971point}.\\

\noindent In the recent papers \cite{jaisson2013limit,jaisson2014fractional}, nearly unstable Hawkes processes are introduced and their asymptotic behavior is investigated. Such processes are defined as Hawkes processes for which the stability condition is almost saturated: $1-|\phi|_{L^1}\ll 1$. The key idea to understand the limiting laws of nearly unstable Hawkes processes is to relate their dynamics with some geometric sums whose parameter tends to zero. In this work, we follow the same strategy in order to derive asymptotic results for nearly unstable AR($\infty$) processes.\\

\noindent We show that the limiting behavior of nearly unstable AR($\infty$) processes strongly depends  on the properties of the sequence $\phi=(\phi_1,\phi_2,\ldots)$. Two main situations have to be considered: the light tail case, where $\phi_i$ goes rapidly to zero (or is even equal to zero for $i$ large enough) and the heavy tail case, where the decay of the sequence $\phi$ is slow. We prove that in the first situation, the cumulated process asymptotically behaves as an integrated Ornstein-Uhlenbeck model, in agreement with the results obtained in the literature for AR($p$) processes. However, in the heavy tail case, a fractional limit is obtained. More precisely, the renormalized cumulated AR($\infty$) process behaves as a kind of integrated fractional Brownian motion with Hurst parameter smaller than $1/2$, or as a process close to a fractional Brownian motion with Hurst parameter larger than $1/2$, depending on the decay rate of the sequence $\phi$.\\

\noindent The paper is organized as follows. In Section \ref{gs}, we introduce the geometric sums and derive their limiting behavior as their parameter tends to zero. We describe our asymptotic framework in Section \ref{asy}. Section \ref{result} contains our main results about the limiting laws of nearly unstable AR($\infty$) processes and their proofs. We also interpret these results in terms of volatility modeling in Section \ref{volat}. Finally, some technical lemmas are relegated to an appendix.

\section{Geometric sums}
\label{gs}

\noindent In this section, we consider the asymptotic behavior of some geometric sums. Understanding these sums is actually the key point in order to be able to derive limiting laws for nearly unstable AR($\infty$) processes.

\subsection{Geometric sums and convolution}

\noindent Let $(a_n)_{n\geq 1}$ and $(\phi_i)_{i\geq 0}$ be two sequences of non-negative numbers such that $a_n<1$, $\phi_0=0$ and
$$\sum_{i=1}^{+\infty} \phi_i =1.$$
In the next sections, $a_n$ will correspond to the sum of the autoregression coefficients and the $\phi_i$ to the (fixed) shape of these coefficients, see Section \ref{asy}.\\

\noindent In this work, geometric sums, see \cite{geosum}, are random variables of the form
$$Y^n=\sum_{i=1}^{I^n} X^i,$$
where $I^n$ is a geometric random variable with parameter $1-a_n$\footnote{$\forall i\geq 0,~ \mathbb{P}[I^n=i]=a_n^{i}(1-a_n)$.} (with the convention $\sum_1^0=0$) and $(X^i)$ is a sequence of strictly positive integer-valued iid random variables such that for any $i\geq 1$, $$\mathbb{P}[X^1=i]=\phi_i.$$

\noindent Now we define the sequence $\psi^n$ by
\begin{equation*}
\psi^n=\sum_{l=0}^{+\infty}(a_n\phi)^{*l},
\end{equation*}
with $\phi^{*0}$ the sequence indexed by $\mathbb{N}$ such that $\phi^{*0}_0=1$ and for $i>0$, $\phi^{*0}_i=0$ and for $n\geq 0$, $\phi^{*(n+1)}=\phi^{*n}\ast \phi$, where $\ast$ denotes the convolution operator for sequences. Note that
\begin{equation*}
\sum_{i=0}^{+\infty}\psi^n_i=\frac{1}{1-a_n}.
\end{equation*}
Actually, the random variable $Y^n$ and the sequence $\psi^n$ are closely related since for any $i\geq 0$,
$$ \mathbb{P}[Y^n=i]=\frac{\psi^n_i}{\sum_{j\geq 0}\psi^n_j}.$$

\noindent We will use the limiting behavior of $Y^n$ to obtain asymptotic properties for the sequence $\psi^n$, see Corollaries \ref{c1} and \ref{c2}. This will provide us key results for our study of nearly unstable AR($\infty$) processes at large time scales. Indeed, in the next sections, for $n$ tending to infinity, we consider on time intervals of the form $[0,nt]$ autoregressive processes whose coefficients sequence can be written $a_n\phi$: $$y^n_k=\varepsilon_k+\sum_{i=1}^{k}a_n\phi_i y^n_{k-i},$$
with $a_n$ tending to one (nearly unstable case). Then, it is easy to obtain the following moving average representation of the process:
$$y^n_k=\sum_{i=0}^{k} \psi^n_{k-i}\varepsilon_i$$ and in particular
$$y^n_{\lfloor nt\rfloor}=\sum_{i=0}^{\lfloor nt\rfloor} \psi^n_{\lfloor nt\rfloor-i}\varepsilon_i.$$
Thanks to the previous representation, we see that the long term behavior of $y^n$ is linked with the asymptotic properties of the $\psi^n_{\lfloor nt\rfloor-i}$ and therefore with those of $Y^n/n$. Indeed, for example, for $nt$ an integer,
$$ \mathbb{P}[Y^n/n=t]=\frac{\psi^n_{nt}}{\sum_{j\geq 0}\psi^n_j}.$$

\noindent We now recall some results from \cite{jaisson2013limit,jaisson2014fractional} which explain the two possible types of asymptotic behaviors for our geometric sums. Exhibiting one behavior or the other depends on the decay rate of the coefficients of $\phi$, that is on the tail of the distribution of $X^1$.

\subsection{Light tail case}
The first asymptotic regime of interest is the one which occurs when the expectation of $X^1$ is finite. Thus we consider the following assumption.
\begin{assumption}
\label{average}
We have
$$\sum_{i=1}^{+\infty}i\phi_i=m<+\infty.$$
\end{assumption}
\noindent In that situation, the ``appropriate'' speed of convergence of $a_n$ towards one for the nearly unstable case is given in the next assumption.
\begin{assumption}
\label{speed}
We have $a_n<1$ and there exists $\lambda>0$ such that
$$n(1-a_n)\underset{n\rightarrow+\infty}{\rightarrow} \lambda.$$
\end{assumption}

\noindent Indeed, recall Proposition 2.2 of \cite{jaisson2013limit}.
\begin{proposition}\label{plt}
Under Assumptions \ref{average} and \ref{speed}, the random variable $Y^n/n$ converges in law towards an exponential random variable with parameter $\lambda/m$.
\end{proposition}

\noindent Using Dini's theorem, this implies the next corollary.
\begin{corollary}
\label{c1}
Under Assumptions \ref{average} and \ref{speed},
$$\sup_{x\in [0,1]}\big|\sum_{i=0}^{\lfloor nx\rfloor} \psi^n_i (1-a_n) - (1-e^{-\frac{\lambda}{m}x})\big|\underset{n\rightarrow +\infty}{\rightarrow} 0.$$
\end{corollary}

\subsection{Heavy tail case}
The second asymptotic regime of interest is the one which occurs when the tail of the distribution of $X^1$ is heavy. This can be formalized as follows.
\begin{assumption}
\label{pl}
There exists $K>0$ and $\alpha\in (0,1)$ such that
$$\sum_{i=N}^{+\infty}\phi_i\underset{N\rightarrow+\infty}{\sim}\frac{K}{N^\alpha}.$$
\end{assumption}

\noindent Let us now write $$\delta=K\Gamma(1-\alpha)/\alpha,$$ with $\Gamma$ the Gamma function.
In that heavy tail case, the ``appropriate'' convergence speed of $a_n$ towards one is given in the next assumption.

\begin{assumption}
\label{speed2}
We have $a_n<1$ and there exists $\lambda>0$ such that
$$n^\alpha (1-a_n)\underset{n\rightarrow+\infty}{\rightarrow} \lambda\delta.$$
\end{assumption}

\noindent Indeed, recall Proposition 2.3 of \cite{jaisson2013limit}.
\begin{proposition}\label{pht}
Under Assumptions \ref{pl} and \ref{speed2}, the random variable $Y^n/n$ converges in law towards the random variable whose Laplace transform is $$\frac{\lambda}{\lambda+z^\alpha}.$$
\end{proposition}

\noindent It is explained in \cite{jaisson2013limit} that this random variable has the density
$$f^{\alpha,\lambda}(x)=\lambda x^{\alpha-1}E_{\alpha,\alpha} (-\lambda x^\alpha),$$
where $E_{\alpha,\alpha}$ is the Mittag-Leffler function, see \cite{mathai2008mittag} for definition.\\

\noindent Writing $$F^{\alpha,\lambda}(x)=\int_0^xf^{\alpha,\lambda}(s)ds$$ and using Dini's theorem, we have the following result.

\begin{corollary}
\label{c2}
Under Assumptions \ref{pl} and \ref{speed2},
$$\sup_{x\in [0,1]}\big|\sum_{i=0}^{\lfloor nx\rfloor} \psi^n_i (1-a_n) - F^{\alpha,\lambda}(x)\big|\underset{n\rightarrow +\infty}{\rightarrow} 0.$$
\end{corollary}

\section{The asymptotic setting}
\label{asy}

As previously explained, we consider a sequence of autoregressive processes of infinite order indexed by $n$. More precisely, $y^n_0=\varepsilon_0$ and for any $k\geq 1$,
$$y^n_k=\varepsilon_k+\sum_{i=1}^{k}\phi^n_i y^n_{k-i},$$ where the $\phi^n_i$ are non-negative coefficients. We assume that the $\varepsilon_k$ do not depend on $n$, are iid centered, and satisfy (for simplicity) $$\E[(\varepsilon_k)^2]=1.$$
We are interested in the long term behavior of the cumulative sums of such autoregressive processes when the stability condition \eqref{stabcond} is almost saturated. To study this, somehow as in \cite{jaisson2013limit,jaisson2014fractional}, we consider that the ``shape" of the coefficients sequence is fixed and that their $L^1$ norm $a_n$ tends to one from below. More precisely, we take
$$\phi^n_i=a_n\phi_i,$$
where $$\sum_{i=1}^{+\infty}\phi_i=1$$
and $a_n\rightarrow 1$ as $n$ tends to infinity, with $a_n<1$. 

\noindent Since we want to investigate the asymptotic properties of cumulated sums of such processes, we introduce the suitably renormalized Donsker line for $t\in [0,1]$:
$$Z^n_t=\frac{1-a_n}{\sqrt{n}}\big(\sum_{j=0}^{\lfloor nt\rfloor} y^n_j+(nt-\lfloor nt\rfloor)y^n_{\lfloor nt\rfloor+1}\big).$$

\section{Main results}
\label{result}

\subsection{The light tail case}


\noindent We give here the behavior of the Donsker line when the autoregression coefficients converge rapidly to zero. Note that the next theorem 
covers in particular the case of AR($p$) processes already studied in the literature.

\begin{theorem}
\label{theo2}
Under Assumptions \ref{average} and \ref{speed}, $(Z^n)$ converges in law towards the process $Z$ defined by
$$Z_t=\int_0^t (1-e^{-\frac{\lambda}{m}(t-s)})dW_s=\int_0^t \int_0^s \frac{\lambda}{m} e^{-\frac{\lambda}{m} (s-u)}dW_u ds,$$
where $W$ is a Brownian motion.
\end{theorem}

\noindent This theorem means that when observed on a time scale of order $1/(1-a_n)$, a cumulated nearly unstable AR($\infty$) process whose coefficients sequence has a light tail behaves as an integrated Ornstein-Uhlenbeck process.\\

\noindent To show this result, the first step is to get the following proposition whose proof can be found in appendix.

\begin{proposition}
\label{tightness}
Under Assumptions \ref{average} and \ref{speed}, the sequence $(Z^n)$ is tight.
\end{proposition}

\noindent Then we need to prove the finite dimensional convergence. Using the moving average representation of the process, we derive
$$Z^n_t=\sum_{i=0}^{\lfloor nt\rfloor} (1-a_n)\Big(\sum_{k=0}^{\lfloor nt\rfloor-i}\psi^n_k +(nt-\lfloor nt\rfloor) \psi^n_{\lfloor nt\rfloor+1-i}\Big)\frac{\varepsilon_i}{\sqrt{n}}+(nt-\lfloor nt\rfloor)(1-a_n)\frac{\varepsilon_{\lfloor nt\rfloor+1}}{\sqrt{n}}.$$
This can be rewritten under the following integral form:
$$Z^n_t=\int_0^{\lfloor nt\rfloor/n} (1-a_n)\Big(\sum_{k=0}^{\lfloor nt\rfloor-\lfloor ns\rfloor}\psi^n_k +(nt-\lfloor nt\rfloor) \psi^n_{\lfloor nt\rfloor+1-\lfloor ns\rfloor}\Big)dW^n_s+(nt-\lfloor nt\rfloor)(1-a_n)\frac{\varepsilon_{\lfloor nt\rfloor+1}}{\sqrt{n}},$$
where 
$$W_s^n=\frac{1}{\sqrt{n}}\sum_{k=0}^{\lfloor ns\rfloor}\varepsilon_k.$$
From Donsker theorem in Skorohod space, we have that $W_s^n$ converges in law towards a Brownian motion for the Skorohod topology. Furthermore, note that 
$$(1-a_n)\Big(\sum_{k=0}^{\lfloor nt\rfloor-\lfloor ns\rfloor}\psi^n_k +(nt-\lfloor nt\rfloor) \psi^n_{\lfloor nt\rfloor+1-\lfloor ns\rfloor}\Big)$$
belongs to 
$$\Big[(1-a_n)\sum_{k=0}^{\lfloor nt\rfloor-\lfloor ns\rfloor}\psi^n_k,(1-a_n)\sum_{k=0}^{\lfloor nt\rfloor+1-\lfloor ns\rfloor}\psi^n_k\Big].$$
Thus, from Corollary \ref{c1}, as a function of $s$, 
$$(1-a_n)\Big(\sum_{k=0}^{\lfloor nt\rfloor-\lfloor ns\rfloor}\psi^n_k +(nt-\lfloor nt\rfloor) \psi^n_{\lfloor nt\rfloor+1-\lfloor ns\rfloor}\Big)$$
tends uniformly to $1-e^{-\frac{\lambda}{m}(t-s)}$ on $[0,t]$. Also, we obviously get that $$(nt-\lfloor nt\rfloor)(1-a_n)\frac{\varepsilon_{\lfloor nt\rfloor+1}}{\sqrt{n}}$$
tends to zero.\\

\noindent Then, using Theorem 2.2 of \cite{kurtz1991weak} together with the fact that Skorohod convergence implies pointwise convergence at continuity points of the limit, for given $t\in [0,1]$, we get the convergence in law of $Z^n_{t}$ towards 
$$\int_0^{t} (1-e^{-\frac{\lambda}{m}(t-s)})d W_s.$$
With the help of Cramer-Wold device, it is easy to extend this result and to show that for any $(t_1,...,t_k)\in [0,1]^k$, we have the convergence in law of $(Z^n_{t_1},\ldots,Z^n_{t_k})$ towards
$$\Big(\int_0^{t_1} (1-e^{-\frac{\lambda}{m}(t_1-s)})d W_s,...,\int_0^{t_k} (1-e^{-\frac{\lambda}{m}(t_k-s)})d W_s\Big).$$
Together with the tightness of $(Z^n)$, this enables us to obtain the weak converges of $(Z^n)$ towards $Z$.


\subsection{The heavy tail case}

Let us now place ourselves under Assumption \ref{pl} which states that the coefficients sequence has a power law type behavior. Then, using the geometric sums interpretation, we see that the natural ``observation scale'' of the process is of order $(1-a_n)^{-1/\alpha}$. This corresponds to Assumption \ref{speed2}. We have the following result.

\begin{theorem}
\label{theo3}
Under Assumptions \ref{pl} and \ref{speed2}, $(Z^n)$ converges in law towards the process $Z$ defined by
$$Z_t=\int_0^t F^{\alpha,\lambda}(t-s)dW_s,$$
where $W$ is a Brownian motion.
\end{theorem}

\noindent In Section \ref{volat}, we will see that for $\alpha>1/2$, the limiting process can be viewed as an integrated rough fractional process, whereas for  $\alpha<1/2$, it is close to a fractional Brownian motion with Hurst parameter larger than $1/2$.\\ 

\noindent To obtain Theorem \ref{theo3}, the same strategy as for the proof of Theorem \ref{theo2} is used. In particular, the following proposition is proved in appendix.
\begin{proposition}
\label{tightness2}
Under Assumptions \ref{pl} and \ref{speed2}, the sequence $(Z^n)$ is tight.
\end{proposition}

\noindent The end the proof of Theorem \ref{theo3} follows as previously, using the decomposition
$$Z^n_t=\int_0^{\lfloor nt\rfloor/n} (1-a_n)\Big(\sum_{k=0}^{\lfloor nt\rfloor-\lfloor ns\rfloor}\psi^n_k +(nt-\lfloor nt\rfloor) \psi^n_{\lfloor nt\rfloor+1-\lfloor ns\rfloor}\Big)dW^n_s+(nt-\lfloor nt\rfloor)(1-a_n)\frac{\varepsilon_{\lfloor nt\rfloor+1}}{\sqrt{n}},$$
the uniform convergence of $$(1-a_n)\sum_{k=0}^{\lfloor nt\rfloor}\psi^n_k$$ towards $F^{\alpha,\lambda}(t)$ (see Corollary \ref{c2}), and the convergence of $W^n_s$ towards a Brownian motion.

\section{Application to volatility modeling}\label{volat}

Let us now interpret Theorem \ref{theo3} in terms of volatility modeling on financial markets. Consider that the log-volatility process is driven at discrete times by a nearly unstable autoregressive process with heavy tailed coefficients sequence. We choose to model the log-volatility rather than the volatility itself because it is well established that it is better approximated by a linear autoregressive process than the volatility, see \cite{andersen2003modeling,bacry2013log,comte1998long,gatheral2014volatility}. Such model can reproduce the clustering property of the volatility at multiple time scales.

\subsection{The case $\alpha>1/2$}

\noindent When $\alpha>1/2$, applying the stochastic Fubini's theorem, see \cite{veraar2012stochastic}, we get that the limiting process for the cumulated sums, $Z$, can be rewritten
$$Z_t=\int_0^t \int_0^s f^{\alpha,\lambda}(s-u)dW_uds.$$
Therefore, in that case, it is differentiable and its derivative $$Y_t=\int_0^t f^{\alpha,\lambda}(t-u)dW_u$$ locally behaves as a fractional diffusion with Hurst parameter $H=\alpha-1/2$. Indeed, $$f^{\alpha,\lambda}(x)\sim c/x^{1-\alpha}$$ when $x$ is close to zero and recall that a fractional Brownian motion $W^H$ can be written as
$$W^H_t=\int_{-\infty}^t \Big[\frac{1}{(t-s)^{1/2-H}}-\frac{1}{(-s)^{1/2-H}_+}\Big]dW_s.$$ In particular, proceeding as in \cite{jaisson2014fractional}, we get that for any $\varepsilon>0$, $Y$ has H\"older regularity $\alpha-1/2-\varepsilon$.\\

\noindent Thus, in this regime, the log-volatility asymptotically behaves as a fractional Brownian motion with Hurst parameter $\alpha-1/2$. According to \cite{gatheral2014volatility}, this is consistent with empirical measures of the smoothness of the volatility process provided that $\alpha\simeq 0.6$.

\subsection{The case $\alpha<1/2$}

\noindent When $\alpha<1/2$, the behavior of $Z$ is quite different. Proceeding as in \cite{jaisson2014fractional}, we get that for any $\varepsilon>0$, $Z$ has H\"older regularity $1/2+\alpha-\varepsilon$. This is not very surprising since the situation $\alpha<1/2$ is close to that of ARFIMA processes\footnote{An ARFIMA process, see \cite{beranlm}, can be written as an infinite order autoregressive process whose sum of the coefficients is equal to one and whose coefficients sequence $\phi$ asymptotically behaves as $\phi_i\sim c/i^{1+\alpha}$ with $0<\alpha<1/2$.}, which are known to behave as a fractional Brownian motion with Hurst parameter $1/2+\alpha$ at large time scales, see \cite{doukhan2003theory}. In this regime, the log-volatility exhibits apparent long memory, as observed for example in \cite{andersen2003modeling}.

\subsection{The case $\alpha=1/2$}

In the critical regime $\alpha=1/2$, we somehow asymptotically retrieve some features of the multifractal model of \cite{bacry2013log}. Indeed, in \cite{bacry2013log}, the ``log-volatility'' $\omega_{l,T}(t)$ is written under the form
$$\omega_{l,T}(t)=\int_{-\infty}^t k_{l,T}(t-s)dW_s,$$
where $W$ is a Brownian motion and $k_{l,T}$ is function behaving in the range $l\ll t \ll T$ as $$k_{l,T}(t)\sim \frac{k_0}{\sqrt{t}},$$ for some model parameters $l$ and $T$.
Therefore, the integrated log-volatility defined as $$\Omega_{l,T}(t)=\int_0^t \omega_{l,T}(s)ds$$ satisfies $$\Omega_{l,T}(t)=\int_{-\infty}^t\big(K_{l,T}(t-s)-K(-s)\big)dW_s,$$
with $$K_{l,T}(t)=\mathbb{I}_{t\geq0}\int_0^tk_{l,T}(s)ds.$$
This behaves as $K_0\sqrt{t}$ in the range $l\ll t \ll T$. \\


\noindent Thus, this ``multifractal" regime ($\alpha=1/2$) appears as the interface between the classical long memory (log-)volatility models ($\alpha<1/2$) and the more recent rough volatility models ($\alpha>1/2$).

\appendix

\section{Appendix}

\noindent In the sequel, $c$ denotes a positive constant which may vary from line to line.

\subsection{Proof of Proposition \ref{tightness}}
In this paragraph, we place ourselves under Assumptions \ref{average} and \ref{speed}. Before proving Proposition \ref{tightness}, we need a technical lemma.


\begin{lemma}
\label{squared}
There exists $c>0$ such that for any $n$, $k$, $k'$,
$$\E[y^n_ky^n_{k'}]\leq \frac{c}{1-a_n}.$$
\end{lemma}

\begin{proof}
For any $k$, using the moving average representation of the process, since $\varepsilon$ is a white noise, we get
$$\E[(y^n_k)^2]=\sum_{j=0}^k (\psi^n_{k-j})^2\leq \sum_{j=0}^{+\infty} (\psi^n_{j})^2.$$
Moreover, using Parseval's Theorem\footnote{In this work, the Fourier transform of a sequence $(f_n)$ is defined for $z\in [-1/2,1/2]$ as $\widehat{f}(z)=\sum_{k\geq 0} f_k e^{-2\pi i k z}$.},
\begin{eqnarray*}
\sum_{j=0}^{+\infty} (\psi^n_{j})^2&=&2\int_0^{1/2} |\widehat{\psi^n}(z)|^2dz\\
&\leq &2\int_0^{1/2} \frac{1}{|1-a_n|^2}dz,
\end{eqnarray*}
where we have used that  
$$|\widehat{\psi^n}(z)|=|\sum_{k\geq 0}(a_n \widehat{\phi}(z))^k|=\frac{1}{|1-a_n\widehat{\phi}(z)|}\leq |\widehat{\psi^n}(0)| =\frac{1}{|1-a_n|}.$$
\end{proof}

\noindent Let us now prove Proposition \ref{tightness}.\\

\noindent The result follows from an application of Kolmogorov's criterion, see \cite{billingsley2009convergence}. We show here that there exists $c>0$ such that for any $n$, $s$ and $t$,
$$\E[|Z^n_t-Z^n_s|^2]\leq c |t-s|^2.$$
We begin with the case where $\lfloor nt\rfloor=\lfloor ns\rfloor$. Using Lemma \ref{squared},
\begin{eqnarray*}
\E[|Z^n_t-Z^n_s|^2]&=& \frac{(1-a_n)^2}{n}n^2 |t-s|^2 \E[(y^n_{\lfloor nt\rfloor+1})^2]\\
&\leq &  cn (1-a_n) |t-s|^2.
\end{eqnarray*}
Using Assumption \ref{speed}, this ends the proof for $\lfloor nt\rfloor=\lfloor ns\rfloor$.\\

\noindent In the case $\lfloor nt\rfloor=\lfloor ns\rfloor+1$,
\begin{eqnarray*}
\E[|Z^n_t-Z^n_s|^2]&\leq& c\big(\E[|Z^n_t-Z^n_{\lfloor nt\rfloor/n}|^2]+\E[|Z^n_{(\lfloor ns\rfloor+1)/n}-Z^n_s|^2]\big)\\
&\leq &  c|t-s|^2+\frac{(1-a_n)^2}{n}(1-ns+\lfloor ns\rfloor)^2\E[(y^n_{\lfloor ns\rfloor+1})^2]\\
&\leq &  c|t-s|^2+\frac{(1-a_n)}{n}(\lfloor nt\rfloor-ns)^2\\
&\leq &  c|t-s|^2.
\end{eqnarray*}

\noindent We now treat the case where $t=k^t/n$ and $s=k^s/n$, where $k^t$ and $k^s$ are integers so that $k^t>k^s$. Using again Lemma \ref{squared} together with Assumption \ref{speed}, we have
\begin{align*}
\E[|Z^n_t-Z^n_s|^2]&= \frac{(1-a_n)^2}{n} \sum_{k_1=k^s+1}^{k^t} \sum_{k_2=k^s+1}^{k^t} \E[y^n_{k_1}y^n_{k_2}]\\
&\leq  \frac{(1-a_n)^2}{n} (k^t-k^s)^2 \frac{c}{1-a_n}\\
& \leq  c(1-a_n)n \Big(\frac{k^t-k^s}{n}\Big)^2\leq c|t-s|^2.
\end{align*}

\noindent Finally, for any $t>s$ so that $\lfloor nt\rfloor\geq \lfloor ns\rfloor+2$, we use the decomposition
$$\E[|Z^n_t-Z^n_s|^2]\leq c\big(\E[|Z^n_t-Z^n_{\lfloor nt\rfloor/n}|^2]+\E[|Z^n_{\lfloor nt\rfloor/n}-Z^n_{(\lfloor ns\rfloor+1)/n}|^2]+\E[|Z^n_{(\lfloor ns\rfloor+1)/n}-Z^n_s|^2]\big).$$
The tightness follows.

\subsection{Proof of Proposition \ref{tightness2}}
In this paragraph, we place ourselves under Assumptions \ref{pl} and \ref{speed2}. Before proving Proposition \ref{tightness2}, we need some technical results.\\

\begin{lemma}
\label{boundphi2}
There exists $c>0$ such that for any $|z|\leq 1/2$,
$$|1-\widehat{\phi}(z)|\geq c |z|^\alpha.$$
\end{lemma}

\begin{proof}
Since Assumption \ref{pl} is satisfied, there exists $c'\neq 0$ and $\delta>0$ such that $$1-\widehat{\phi}(z)\underset{z\rightarrow 0}{\sim}c'z^\alpha$$ and $$\forall |z|\leq \delta,~ |1-\widehat{\phi}(z)|\geq |c'||z|^\alpha/2.$$ 
On $[\delta,1/2]$, $z\mapsto |1-\widehat{\phi}(z)|/|z|^\alpha$ is continuous and therefore has a minimum attained in some $z_0$: $|1-\widehat{\phi}(z_0)|/|z_0|^\alpha=M$.\\

\noindent $M$ is strictly positive since $\text{Re}(\widehat{\phi}(z_0))=\E[\cos(2\pi z_0 X^1)]$, which is strictly smaller than one because, using Assumption \ref{pl}, $z_0 X^1$ does not almost surely belong to $\mathbb{N}$.\\

\noindent  Taking $c=\min(|c'|/2,M)$ ends the proof.
\end{proof}

\begin{lemma}
\label{squared2}
For $\alpha > 1/2$, there exists $c>0$ such that for any $n$, $k$, $k'$,
$$\E[y^n_ky^n_{k'}]\leq c(1-a_n)^{1/\alpha-2}.$$
\end{lemma}

\begin{proof}
As before, we have
$$\E[(y^n_k)^2]\leq 2\int_0^{1/2} \frac{1}{|1-a_n\widehat{\phi}(z)|^2}dz.$$
Therefore, for $n$ large enough, $a_n>1/2$ and $$|1-a_n\widehat{\phi}(z)|^2=|1-a_n+a_n(1-\widehat{\phi}(z))|^2 \geq |1-a_n|^2+ \frac{|1-\widehat{\phi}(z)|^2}{4}.$$
Using Lemma \ref{boundphi2}, this implies that
$$\E[(y^n_k)^2]\leq   \int_0^{1/2} \frac{c}{|1-a_n|^2+|z|^{2\alpha}}dz.$$
Therefore,
$$\E[(y^n_k)^2]\leq \int_0^{(1-a_n)^{1/\alpha}} \frac{c}{|1-a_n|^2}dz+\int_{(1-a_n)^{1/\alpha}}^{1/2} \frac{c}{|z|^{2\alpha}}dz,$$
which ends the proof.
\end{proof}

\begin{lemma}
\label{squared3}
For $\alpha < 1/2$, there exists $c>0$ such that for any $n$, $k$, $k'$,
\begin{equation*}
\E[y^n_ky^n_{k'}]\leq \left\{
    \begin{array}{ll}
        c(1-a_n)^{2\alpha-1} & \mbox{if } k=k' \\
        c |k-k'|^{2\alpha-1} & \mbox{if } k\neq k'.
    \end{array}
\right.
\end{equation*}
\end{lemma}

\begin{proof}
The fact that $$\E[(y^n_k)^2]\leq c(1-a_n)^{2\alpha-1}$$ is a direct consequence of the proof of Lemma \ref{squared2} together with the inequality
$1-2\alpha\geq 2-1/\alpha$ for $\alpha<1/2$.\\

\noindent To prove the other inequality, remark that using the moving average representation of the process, for $k\geq k'$, we get
$$\E[y^n_k y^n_{k'}]=\sum_{j=0}^{k'} \psi^n_{k-j}\psi^n_{k'-j} \leq \sum_{j=0}^{+\infty} \psi^n_{j}\psi^n_{j+k-k'}.$$
Therefore, using Parseval's theorem together with Lemma \ref{boundphi2} and the fact that the $\psi_j$ are non-negative, we obtain
$$\E[y^n_k y^n_{k'}]\leq c \text{Re}\Big(\int_0^{1/2} \frac{e^{2\pi i |k-k'|z}}{|1-a_n\widehat{\phi}(z)|^2}dz\Big)\leq c \text{Re}\Big(\int_0^{1/2} \frac{e^{2\pi i |k-k'|z}}{|z|^{2\alpha}}dz\Big).$$
Thus, using Abel's theorem, we finally obtain
$$\E[y^n_k y^n_{k'}]\leq \frac{c}{|k-k'|^{1-2\alpha}}.$$

\end{proof}

\begin{lemma}
\label{squared4}
For $\alpha=1/2$, for any $\varepsilon\in(0,1)$, there exists $c>0$ such that for any $n$, $k$, $k'$,
\begin{equation*}
\E[y^n_ky^n_{k'}]\leq \left\{
    \begin{array}{ll}
        c(1-a_n)^{-\varepsilon} & \mbox{if } k=k' \\
        c +c\big((1-a_n)^2|k-k'|\big)^{-\varepsilon}\ & \mbox{if } k\neq k'.
    \end{array}
\right.
\end{equation*}
\end{lemma}
\begin{proof}
Using the same proof as for Lemma \ref{squared2}, we get
$$\E[(y^n_k)^2]\leq c|\text{log}(1-a_n)|.$$
The first inequality follows. We obtain the second inequality using that for $k>k'$, 
\begin{align*}
\E[y^n_k y^n_{k'}]&\leq c\text{Re}\Big(\int_0^{1/2} \frac{e^{2\pi i |k-k'|z}}{(1-a_n)^2+|z|}dz\Big)\\
&\leq c\int_0^{(1-a_n)^2} \frac{1}{(1-a_n)^2}dz+c\text{Re}\Big(\int_{(1-a_n)^2}^{1/2} \frac{e^{2\pi i |k-k'|z}}{|z|}dz\Big)\\
&\leq c+c\big|\text{log}\big((1-a_n)^2|k-k'|\big)\big|.
\end{align*}
\end{proof}

\noindent Let us now prove Proposition \ref{tightness2}.\\

\noindent In the case where $\alpha> 1/2$, the proof that 
$$\E[|Z^n_t-Z^n_s|^2]\leq c |t-s|^2$$
is almost the same as in the light tail case replacing the use of Lemma \ref{squared} by that of Lemma \ref{squared2}.\\

\noindent The case $\alpha<1/2$ is slightly more complicated. We now show that 
$$\E[|Z^n_t-Z^n_s|^2]\leq c |t-s|^{1+\eta},$$ 
for some $\eta>0$.
As before, we begin with the case where $\lfloor nt\rfloor=\lfloor ns\rfloor$. Using Lemma \ref{squared3} together with Assumption \ref{speed2} and the fact that $1-\alpha(2\alpha+1)\geq 0$, we get
\begin{align*}
\E[|Z^n_t-Z^n_s|^2]&= \frac{(1-a_n)^2}{n}n^2 |t-s|^2 \E[(y^n_{\lfloor nt\rfloor +1})^2]\\
&\leq   c\frac{(1-a_n)^2}{n}n^2 |t-s|^{1+\alpha(2\alpha+1)}\frac{1}{n^{1-\alpha(2\alpha+1)}}(1-a_n)^{2\alpha-1}\\
&\leq  c|t-s|^{1+\alpha(2\alpha+1)}.
\end{align*}

\noindent Using the same arguments as for the light tail case, we get a similar bound for $\lfloor nt\rfloor=\lfloor ns\rfloor+1$.\\

\noindent We now treat the case where $t=k^t/n$ and $s=k^s/n$, where $k^t$ and $k^s$ are integers so that $k^t>k^s$. Using again Lemma \ref{squared3}, we have
$$
\E[|Z^n_t-Z^n_s|^2]=\frac{(1-a_n)^2}{n} \sum_{k_1=k^s+1}^{k^t} \sum_{k_2=k^s+1}^{k^t} \E[y^n_{k_1}y^n_{k_2}].$$
This is smaller than (with obvious notation)
$$\frac{(1-a_n)^2}{n} \Big(\frac{c(k^t-k^s)}{(1-a_n)^{1-2\alpha}}+c\sum_{\Delta k=1}^{k^t-k^s-1} \frac{1}{(\Delta k)^{1-2\alpha}}\# \{(k_1,k_2)\in  [k^s+1,k^t];|k_1-k_2|=\Delta k\}\Big).$$
Therefore, using Assumption \ref{speed2}, we get
\begin{eqnarray*}
\E[|Z^n_t-Z^n_s|^2]&\leq & \frac{(1-a_n)^2}{n} \Big(\frac{c(k^t-k^s)}{(1-a_n)^{1-2\alpha}}+c\sum_{\Delta k=1}^{k^t-k^s-1} \frac{1}{(\Delta k)^{1-2\alpha}} (k^t-k^s)\Big)\\
& \leq & \frac{(1-a_n)^2}{n} \Big(\frac{c(k^t-k^s)}{(1-a_n)^{1-2\alpha}}+ c(k^t-k^s)^{1+2\alpha}\Big)\\
&\leq & c((t-s)^{1+\alpha+2\alpha^2}+(t-s)^{1+2\alpha}).
\end{eqnarray*}
The result for any $s$ and $t$ in $[0,1]$ is obtained as in the light tail case.\\

\noindent Finally, the proof for $\alpha=1/2$ is obtained the same way, using Lemma \ref{squared4} instead of Lemma \ref{squared3}.

\bibliographystyle{abbrv}
\bibliography{bibli_AR}

\end{document}